\documentclass[a4paper,11pt]{article}
\usepackage[T1]{fontenc}
\usepackage[utf8]{inputenc}
\usepackage[english]{babel}
\usepackage{amsmath,amssymb,amsthm}
\usepackage{enumerate}
\usepackage{indentfirst}
\usepackage[hidelinks,linktocpage=true]{hyperref}
\usepackage{url}
\hypersetup{
  pdftitle={A finite forbidden family with superlinear surplus and no three-factor product extremizers},
  pdfauthor={Chuandong Xu}
}
\newtheorem{thm}{Theorem}[section]

\newtheorem{problem}[thm]{Problem}

\newtheorem{lem}[thm]{Lemma}

\newtheorem{cor}[thm]{Corollary}
\newtheorem{pro}[thm]{Proposition}
\newcommand{\ex}{\operatorname{ex}}
\newcommand{\EX}{\operatorname{EX}}
\newcommand{\Free}{\operatorname{Free}}

\title{\textbf{A finite forbidden family with superlinear surplus and
no three-factor product extremizers}}
\author{Chuandong Xu\footnote{School of
Mathematics and Statistics, Xidian University, Xi'an, 710071, China.
\href{mailto:xuchuandong@xidian.edu.cn}{\nolinkurl{xuchuandong@xidian.edu.cn}}.
Supported by the National Natural Science
Foundation of China (No. 12571378).}}
\date{\today}
\begin{document}
\maketitle
\begin{abstract}
We construct a fixed finite family $\mathcal L$ of ordinary forbidden
subgraphs with $p(\mathcal L)=3$ and a constant $c>0$ such that
\[
   \ex(n,\mathcal L)>t_3(n)+cn^{3/2}
\]
at every sufficiently large order.  Nevertheless, the complement of
every sufficiently large $\mathcal L$-extremal graph has at most two
connected components.  In particular, no such extremal graph is a
complete join of three graphs of positive order.  This gives a negative answer to a natural
existence-only question motivated by the Simonovits Product Conjecture, in which one
asks only for one product extremizer at each sufficiently large order.
\end{abstract}


\section{Introduction}

All graphs in this paper are finite and simple and are considered up to
isomorphism; containment always means ordinary, not necessarily induced,
subgraph containment.  For a nonempty family
$\mathcal F$, let $\ex(n,\mathcal F)$ be the maximum number of edges in an
$n$-vertex $\mathcal F$-free graph and let $\EX(n,\mathcal F)$ be the set of
extremal graphs; at order zero, these have their usual empty-graph meanings.
We write
\[
 p(\mathcal F)=\min_{F\in\mathcal F}\chi(F)-1
 \quad\hbox{and}\quad
 t_r(n)=e(T_{n,r}),
\]
where $T_{n,r}$ is the balanced complete $r$-partite graph.  The Erd\H{o}s--Stone--Simonovits theorem gives
$\ex(n,\mathcal F)=t_{p(\mathcal F)}(n)+o(n^2)$ whenever
$p(\mathcal F)\geq 2$ \cite{ES46,ES66}; the structure hidden in the error term is the subject
of the Simonovits Product Conjecture and its variants
\cite{ES46,FS13,Simonovits83,Simonovits97,SS19}.

For vertex-disjoint graphs $G_1,\ldots,G_r$, their \emph{complete join}
$G_1\otimes\cdots\otimes G_r$ is obtained by adding every edge between
distinct factors.  If $\kappa(H)$ denotes the number of connected components
of $H$, then a graph $G$ admits such a representation with $r$ factors of
positive order precisely when
\[
 G\cong G_1\otimes\cdots\otimes G_r
 \text{ for some $G_1,\ldots,G_r$ with $v(G_i)>0$}
 \quad\Longleftrightarrow\quad
 \kappa(\overline G)\geq r.
 \tag{1.1}
\]
The usual product statements ask for considerably more than (1.1): the
factors should come from fixed local extremal classes, and often every
extremal graph should admit such a representation.  F\"uredi and Simonovits
suggested in Remark~2.9(c) of \cite{FS13} (see also
Conjecture~1.2 in \cite{Xu26}) that, even when uniqueness fails,
there might still be a product extremizer at each sufficiently large order.
Because the local families in that remark are left informal, we isolate only
its profile-free existence content.

\begin{problem}[F\"{u}redi and Simonovits \cite{FS13}]\label{prob:existence}
Let $\mathcal F$ be a fixed nonempty finite family with $p=p(\mathcal F)\geq2$.
Suppose that for some $c>0$ and $0<\varepsilon<1$,
\[
 \ex(n,\mathcal F)>t_p(n)+c n^{1+\varepsilon}
\]
for every sufficiently large $n$.  Must every sufficiently large order admit
some $G\in\EX(n,\mathcal F)$ with $\kappa(\overline G)\geq p$?
\end{problem}

The distinction between ``some'' and ``every'' is substantive.  Earlier
constructions produce nonproduct extremizers with bounded surplus; see, for
example, \cite{GST98,PSY26,Simonovits83}.  The superlinear construction in
\cite{Xu26} gives a nonproduct extremizer at every sufficiently large order
but also has a product equality branch.  Problem~\ref{prob:existence} asks
whether some product branch must always survive.  Our main theorem answers
this precise question negatively.

\begin{thm}\label{thm:main}
There exist a fixed nonempty finite family $\mathcal L_{180}$, a constant
$c>0$, and an integer $N$ such that $p(\mathcal L_{180})=3$ and, for every $n\geq N$,
\[
 \ex(n,\mathcal L_{180})>t_3(n)+c n^{3/2},
 \tag{1.2}
\]
whereas every $G\in\EX(n,\mathcal L_{180})$ satisfies
\[
 \kappa(\overline G)\leq2.
 \tag{1.3}
\]
In particular, no $\mathcal L_{180}$-extremal graph of sufficiently large
order is a complete join of three graphs of positive order.
\end{thm}

The family is defined through a hereditary class specified by a vertex
partition.  We use
$K_{s_1,\ldots,s_r}$ for the complete $r$-partite graph with the
indicated part sizes and $I_s$ for the edgeless graph on $s$ vertices.  Put
\[
 \mathcal A_S=
 \{K_{1,3,3},\ K_3\otimes I_3,\ (2K_2)\otimes I_3,\ K_4\}
 \tag{1.4}
\]
and define $\mathcal C$ to be the class of graphs $G$ for which there is a
partition $V(G)=R\sqcup B$ such that $G[R]$ is $\mathcal A_S$-free and
$G[B]$ is $C_4$-free; edges between $R$ and $B$ are unrestricted.  Thus
$\mathcal C$ is closed under taking ordinary subgraphs.  Let
$\mathcal L(\mathcal C)$ be the family of graphs outside $\mathcal C$ that
are minimal under the ordinary subgraph relation, and define
\[
 \mathcal L_q=\{L\in\mathcal L(\mathcal C):v(L)\leq q\}.
 \tag{1.5}
\]
Then $\mathcal C=\Free(\mathcal L(\mathcal C))$; the family
$\mathcal L(\mathcal C)$ need not be finite.  We take $\mathcal L_{180}$
from (1.5).

Two ingredients make the finite truncation useful.  First, on a complete
join with three factors of order at least four, nonmembership in $\mathcal C$
always has a certificate on at most $180$ vertices.  Hence an
$\mathcal L_{180}$-free graph with such a presentation already belongs to
$\mathcal C$.  Second, joining any nonempty graph to a certain ten-vertex
complete multipartite graph produces a graph outside $\mathcal C$.  These
facts divide the proof into the case of three large join factors, which is
forced into $\mathcal C$, and the case of a bounded join factor, which has only
linear surplus.  Both contradict extremality.

Section~\ref{sec:aux} establishes the extremal formula, the bounded
obstruction criterion, and the ten-vertex obstruction.  The proof of the
main theorem is completed in Section~\ref{sec:mainproof}.

\paragraph{Statement of AI use.}
The proof strategy and initial manuscript draft were generated
by GPT-5.6 Sol in response to a problem formulated by the author.  The author
verified and revised the manuscript and assumes responsibility for all
content.

\section{Auxiliary results}\label{sec:aux}

\subsection{Classical extremal inputs}

We use three classical extremal inputs.  The first is the following exact
result of Simonovits \cite[Proposition~2]{Simonovits83} for the family
$\mathcal A_S$.

\begin{thm}[Simonovits]\label{thm:Simonovits-input}
For all sufficiently large $m$,
\[
 \ex(m,\mathcal A_S)=t_2(m)+3.
 \tag{2.1}
\]
Every extremal graph is obtained from a balanced complete bipartite graph by
placing one copy of $P_3$ in each side and deleting the cross edge between
the two path centres.  In particular, the complement of every sufficiently
large $\mathcal A_S$-extremal graph is connected.
\end{thm}

The second input is the case of the Erd\H{o}s--Simonovits theorem on
complete multipartite forbidden graphs in which the smallest part has order
one \cite{ES72}.  Set
\[
 T=K_{1,3,3,3}.
 \tag{2.2}
\]
Then
\begin{equation}
 \ex(m,T)=t_3(m)+O(m).
 \tag{2.3}
\end{equation}

Finally, we use the classical estimate
\[
 \ex(m,C_4)=\frac12m^{3/2}+o(m^{3/2})
 \tag{2.4}
\]
from \cite{Brown66,ERS66,KST54}; see also
\cite[Theorem~3.4]{FS13}.

\subsection{The hereditary class and its finite obstruction family}

Since $\mathcal C$ is closed under ordinary subgraphs, its minimal
nonmembers $\mathcal L(\mathcal C)$ are well defined, each $\mathcal L_q$
is finite, and
\[
 \mathcal C\subseteq\Free(\mathcal L_q).
 \tag{2.5}
\]

\begin{pro}\label{prop:parameter-lower}
For every $q\geq16$, the family $\mathcal L_q$ is nonempty and finite,
\[
 p(\mathcal L_q)=3,
 \tag{2.6}
\]
and there is a constant $c>0$ such that
\[
 \ex(n,\mathcal L_q)>t_3(n)+cn^{3/2}
 \tag{2.7}
\]
for every sufficiently large $n$.
\end{pro}

\begin{proof}
Every $3$-colourable graph belongs to $\mathcal C$: put two colour classes
on the red side, which is bipartite and hence $\mathcal A_S$-free, and the
third independent class on the blue side.  Hence every member of
$\mathcal L(\mathcal C)$ has chromatic number at least four.

On the other hand, $K_{4,4,4,4}\notin\mathcal C$.  Indeed, if a blue set
contains two vertices in each of two partite sets, it contains a $K_{2,2}$.
Otherwise at least three partite sets contain three red vertices each, and
the red graph contains $K_{1,3,3}$.  Choose a subgraph $L_0$ of
$K_{4,4,4,4}$ that is minimal outside $\mathcal C$.  Then
$4\leq\chi(L_0)\leq4$, so $L_0\in\mathcal L_q$ and (2.6) follows.

For (2.7), start with a balanced complete three-partite graph, choose one
part $W$, and place a graph $Q\in\EX(|W|,C_4)$ inside $W$, leaving the other
two parts internally empty.  Put the latter two parts on the red side and
$V(Q)$ on the blue side.
The resulting graph lies in $\mathcal C$, and hence is $\mathcal L_q$-free,
and has $t_3(n)+e(Q)$ edges.  Estimate (2.4), applied to a part of order
$n/3+O(1)$, proves (2.7).
\end{proof}

\begin{lem}\label{lem:convolution}
For every $n$, with the maximum taken over nonnegative integers $u,w$,
\[
 \ex(n,\mathcal L(\mathcal C))=
 \max_{\substack{u,w\geq0\\u+w=n}}
 \{\ex(u,\mathcal A_S)+uw+\ex(w,C_4)\}.
 \tag{2.8}
\]
Moreover, if $G\in\EX(n,\mathcal L(\mathcal C))$ and $R\sqcup B$ witnesses
$G\in\mathcal C$, then all $R$--$B$ edges are present and
\[
 G=A_u\otimes Q_w,
 \quad A_u\in\EX(u,\mathcal A_S),
 \quad Q_w\in\EX(w,C_4),
 \tag{2.9}
\]
where $(u,w)$ attains the maximum in (2.8).
\end{lem}

\begin{proof}
For a fixed split of orders $u,w$, completing all cross edges preserves
membership in $\mathcal C$.  Each side may then be replaced independently
by a same-order factor extremizer.  This proves both assertions.
\end{proof}

\subsection{A bounded obstruction criterion for complete joins with three factors}

The next lemma is the reason that a finite truncation of the generally
infinite obstruction set is sufficient for our argument.  Recall that
\[
 h=\max_{A\in\mathcal A_S}v(A)=7.
 \tag{2.10}
\]

\begin{lem}\label{lem:180}
Let $V(G)=V_1\sqcup V_2\sqcup V_3$, where $|V_i|\geq4$, and suppose every
edge between distinct factors is present.  Put
$R_k=\bigcup_{i\neq k}V_i$.  Then $G\in\mathcal C$ if and only if, for some
$k\in\{1,2,3\}$, one of the following conditions holds:
\begin{enumerate}[(a)]
\item $G[V_k]$ is $C_4$-free and $G[R_k]$ is
      $\mathcal A_S$-free;
\item $G[V_k]$ is $P_3$-free and some $x\in R_k$ belongs to every copy of
      a member of $\mathcal A_S$ in $G[R_k]$.
\end{enumerate}
If $G\notin\mathcal C$, then $G$ has an induced nonmember of $\mathcal C$
on at most
\[
 3h(h+1)+12=180
 \tag{2.11}
\]
vertices.
\end{lem}

\begin{proof}
Interpret membership in $\mathcal C$ as a valid red--blue vertex partition.
Two factors cannot each contain two blue vertices, since their cross edges
give a blue $K_{2,2}=C_4$.  Hence two factors each contain at least three
red vertices.  If the third factor contains a red vertex, these vertices
give a red $K_{1,3,3}$.  Thus one factor, say $V_k$, is entirely blue.  There is at
most one further blue vertex in $R_k$, again by the $K_{2,2}$ obstruction.

If there is no blue vertex in $R_k$, condition (a) is necessary and sufficient.  If the
unique exception is $x\in R_k$, the red condition is equivalent to $x$
belonging to every $\mathcal A_S$-copy in $G[R_k]$.  Since $x$ is complete
to $V_k$, the graph $G[V_k\cup\{x\}]$ is $C_4$-free exactly when
$G[V_k]$ is $P_3$-free.  This proves the characterization in both directions.

Suppose now that $G\notin\mathcal C$.  For each $k\in\{1,2,3\}$ choose a witness to the
failure of the two conditions.  If $G[V_k]$ contains a $C_4$, choose four
vertices of such a copy.  If $G[V_k]$ is $C_4$-free but contains a $P_3$,
then $G[R_k]$ contains a member of $\mathcal A_S$; include that copy and the
path, using at most $h+3$ vertices.  Finally, if $G[V_k]$ is $P_3$-free,
the family of $\mathcal A_S$-copies in $G[R_k]$ has empty total
intersection.  Choose one $\mathcal A_S$-copy $Q_0$ and, for every $v\in V(Q_0)$, choose a
copy $Q_v$ avoiding $v$.  Their union has at most
$h(h+1)=56$ vertices and still has empty total intersection.

Take the union of the three chosen witnesses and, if necessary, add
arbitrary vertices so that each factor contains four chosen vertices.  The
induced graph obtained has at most $3\cdot56+12=180$ vertices.  It fails both
conditions for every $k$.  By the characterization already
proved, this induced graph is not in $\mathcal C$.
\end{proof}

\begin{cor}\label{cor:finite-obstruction}
Under the hypotheses of Lemma~\ref{lem:180},
\[
 G\text{ is }\mathcal L_{180}\text{-free}
 \quad\Longrightarrow\quad G\in\mathcal C.
 \tag{2.12}
\]
\end{cor}

\begin{proof}
If $G\notin\mathcal C$, Lemma~\ref{lem:180} gives a nonmember $H$ on at
most $180$ vertices.  By repeatedly deleting vertices and edges while
nonmembership persists, $H$ contains a subgraph-minimal nonmember
$L\in\mathcal L_{180}$, contrary to the $\mathcal L_{180}$-freeness of $G$.
\end{proof}

\subsection{Extremizers and a universal obstruction}

\begin{pro}\label{prop:Cext}
For all sufficiently large $n$, every
$G\in\EX(n,\mathcal L(\mathcal C))$ satisfies
\[
 \kappa(\overline G)=2.
 \tag{2.13}
\]
\end{pro}

\begin{proof}
By Lemma~\ref{lem:convolution}, write $G=A_u\otimes Q_w$ as in (2.9).
The split $u=2n/3+O(1)$ gives
$\ex(n,\mathcal L(\mathcal C))\geq n^2/3-o(n^2)$.
Suppose that, along a sequence of maximizing pairs with $n\to\infty$, one
has $u/n\to0$.  Then the trivial bound
$\ex(u,\mathcal A_S)\leq\binom u2$, together with (2.4), makes the
right-hand side of (2.8) $o(n^2)$, a contradiction.  Hence $u=\Omega(n)$,
so Theorem~\ref{thm:Simonovits-input} applies for all sufficiently large
$n$.  If some subsequence satisfied $w/n\to0$, then that theorem would make
the right-hand side of (2.8) $n^2/4+o(n^2)$, again contradicting the same
comparison split.  Consequently $w=\Omega(n)$.  Now
Theorem~\ref{thm:Simonovits-input} and (2.4) show that the maximizing orders satisfy
\[
 u=\frac{2n}{3}+o(n),\qquad w=\frac n3+o(n).
 \tag{2.14}
\]
Indeed,
\[
 \frac{u^2}{4}+u(n-u)
 =\frac{n^2}{3}-\frac34\left(u-\frac{2n}{3}\right)^2,
 \tag{2.15}
\]
and comparison with $u=2n/3+O(1)$ proves (2.14).  Thus both factors have
sufficiently large order.  By Theorem~\ref{thm:Simonovits-input},
$\overline{A_u}$ is connected.

We claim that $\overline{Q_w}$ is connected as well.  Otherwise $Q_w$ is a
join of at least two nonempty graphs.  Since $Q_w$ is $C_4$-free, at most one
join factor has more than one vertex.  If all factors are singletons, then
$C_4$-freeness forces $w\leq3$; otherwise there can be only one additional
singleton factor.  Hence, at large order, $Q_w=K_1\otimes R$.  If $R$ contains a
$P_3$, its three vertices together with the universal vertex contain an
ordinary $C_4$.  Thus $R$ is $P_3$-free and has maximum degree at most one,
so $e(Q_w)=O(w)$.  This contradicts the lower bound
$\ex(w,C_4)=\Omega(w^{3/2})$.  Therefore both factor complements are connected,
and
\[
 \overline G=\overline{A_u}\sqcup\overline{Q_w}
\]
has exactly two components.
\end{proof}

The case with a bounded join factor will be excluded by the following finite
obstruction.  Let $T=K_{1,3,3,3}$ as in (2.2), and denote its partite sets by
$P_0,P_1,P_2,P_3$, of orders $1,3,3,3$ respectively.

\begin{lem}\label{lem:universal-obstruction}
For every nonempty graph $A$,
\[
 A\otimes T\notin\mathcal C.
 \tag{2.16}
\]
\end{lem}

\begin{proof}
Since $\mathcal C$ is closed under taking subgraphs, it is enough to show
$K_1\otimes T\notin\mathcal C$.  Denote the vertex of the $K_1$ factor by
$z$ and suppose, for a contradiction, that a valid red--blue partition
exists.

First suppose $z$ is red.  A $C_4$-free blue set in $T$ has one of the
following forms: it is empty; it meets only one partite set; it meets two
partite sets and at least one of them in exactly one vertex; or it meets
three partite sets in exactly one vertex each.  If it meets at most one
partite set, or if it meets $P_0$ and only one large partite set, then two
untouched large partite sets already form a red $K_{3,3}$.  In all remaining
cases, write the numbers of red vertices in $(P_0,P_1,P_2,P_3)$.  Up to a
permutation of the three large partite sets, this vector componentwise
dominates one of
\[
 (1,2,0,3),\qquad (0,2,2,3),\qquad (1,2,2,2),
\]
for which one may group the red vertices as $1+2$ versus $3$, $2+2$
versus $3$, and $1+2$ versus $2+2$, respectively.  Thus in every case the
red vertices of $T$ contain a $K_{3,3}$.  Adding the red vertex $z$
therefore gives a red $K_{1,3,3}$, a contradiction.

Now suppose $z$ is blue.  The blue vertices of $T$ must induce a
$P_3$-free graph, since $z$ together with a blue $P_3$ forms a $C_4$.
Consequently they either lie in one partite set or consist of one vertex
from each of two partite sets.  In the first case the remaining red vertices of $T$
contain $K_{1,3,3}$.  In the second, if one occupied partite set is $P_0$, use the
two untouched large partite sets as the two three-vertex sides and a red vertex of
the touched large partite set as the apex.  If two large partite sets are occupied, use
the untouched large partite set as one side, $P_0$ together with the two remaining
vertices of one touched partite set as the other side, and a remaining vertex of
the other touched partite set as the apex.  Again the red graph contains
$K_{1,3,3}$, a contradiction.
\end{proof}

\section{Excluding every three-factor product extremizer}\label{sec:mainproof}

We now combine the preceding ingredients.

\begin{proof}[Proof of Theorem~\ref{thm:main}]
Set $\mathcal L_{180}$ as in (1.5).  Proposition~\ref{prop:parameter-lower}
gives finiteness, $p(\mathcal L_{180})=3$, and (1.2).  Also
$\mathcal C\subseteq\Free(\mathcal L_{180})$.

Fix a sufficiently large $n$ and let
$G\in\EX(n,\mathcal L_{180})$.  Suppose for a contradiction that
$\kappa(\overline G)\geq3$.  We split according to the sizes of the connected
components of $\overline G$.

Assume first that these components can be grouped into three collections,
each of total order at least four.  Their vertex unions form three join factors
between which every edge of $G$ is present.  Corollary~\ref{cor:finite-obstruction}
therefore gives $G\in\mathcal C$.  Since
$\mathcal L_{180}\subseteq\mathcal L(\mathcal C)$ and $G$ is
$\mathcal L_{180}$-extremal,
\[
 \ex(n,\mathcal L(\mathcal C))
 \leq e(G)=\ex(n,\mathcal L_{180})
 \leq \ex(n,\mathcal L(\mathcal C)).
 \tag{3.1}
\]
Thus $G\in\EX(n,\mathcal L(\mathcal C))$, contradicting
Proposition~\ref{prop:Cext}.

It remains to consider the case in which no such grouping is possible.
Call a component of $\overline G$ \emph{large} if it has at least four
vertices, and let $A$ be the graph induced by the union of the remaining
components.  There are at most two large components.  If there are two, the
total order of the small components is at most three.  If there is exactly one
large component and the small components had total order at least ten, then
greedily collecting them would produce one group of order between four and
six, leaving another group of order at least four.  This again contradicts
the assumed failure of grouping.  Hence, in this case, the small components
have total order at most nine.
If there is no large component and $n\geq16$, two successive greedy groups
of orders between four and six leave at least four vertices.  This produces
the forbidden grouping.  Since
$\kappa(\overline G)\geq3$, the union of the small components is nonempty,
and consequently
\[
 G=A\otimes X,\qquad 1\leq a:=v(A)\leq9,
 \tag{3.2}
\]
where $X$ is the subgraph of $G$ induced by the union of the one or two large
components of $\overline G$.

If $X$ contains $T$, choose one vertex of $A$.  Lemma~\ref{lem:universal-obstruction}
shows that the resulting copy of $K_1\otimes T$ is not in $\mathcal C$.
The $11$-vertex graph $K_1\otimes T$ therefore contains a member of
$\mathcal L(\mathcal C)$ on at most eleven vertices, and hence a member of
$\mathcal L_{180}$.  This contradicts
the $\mathcal L_{180}$-freeness of $G$.  Hence $X$ is $T$-free.

Using (2.3) and the fact that $a\leq9$, we obtain
\[
 \begin{aligned}
 e(G)
 &\leq \binom a2+a(n-a)+\ex(n-a,T)\\
 &\leq \binom a2+a(n-a)+t_3(n-a)+O(n)\\
 &=t_3(n)+O(n).
 \end{aligned}
 \tag{3.3}
\]
For large $n$ this contradicts (1.2).  Therefore
$\kappa(\overline G)\leq2$, proving (1.3) and the theorem.
\end{proof}

\subsection*{Concluding remarks}

The proof uses the finite family $\mathcal L_{180}$ only to force membership
in $\mathcal C$ when three join factors of order at least four are already
present.  It does not assert that every
$\mathcal L_{180}$-extremal graph belongs to $\mathcal C$, nor does it
classify those extremal graphs or determine $\ex(n,\mathcal L_{180})$
exactly.  The
integer $180$ is the transparent certificate bound
$3\cdot7\cdot8+12$; no optimality is claimed.

The construction also does not produce a single forbidden graph with the
same property.  Replacing the finite obstruction family by one graph while
preserving the conclusion for every extremal graph is a separate problem.


\begin{thebibliography}{99}

\bibitem{Brown66}
W. G. Brown,
On graphs that do not contain a Thomsen graph,
\emph{Canad. Math. Bull.} 9 (1966), 281--285.

\bibitem{ERS66}
P. Erd\H{o}s, A. R\'{e}nyi, and V. T. S\'{o}s,
On a problem of graph theory,
\emph{Studia Sci. Math. Hungar.} 1 (1966), 215--235.

\bibitem{ES46}
P. Erd\H{o}s and A. H. Stone,
On the structure of linear graphs,
\emph{Bull. Amer. Math. Soc.} 52 (1946), 1087--1091.

\bibitem{ES66}
P. Erd\H{o}s and M. Simonovits,
A limit theorem in graph theory,
\emph{Studia Sci. Math. Hungar.} 1 (1966), 51--57.

\bibitem{ES72}
P. Erd\H{o}s and M. Simonovits,
An extremal graph problem,
\emph{Acta Math. Acad. Sci. Hungar.} 22 (1971), 275--282.

\bibitem{FS13}
Z. F\"{u}redi and M. Simonovits,
The history of degenerate (bipartite) extremal graph problems,
in \emph{Erd\H{o}s Centennial}, Bolyai Soc. Math. Stud. 25,
Springer, 2013, 169--264.

\bibitem{GST98}
J. R. Griggs, M. Simonovits, and G. R. Thomas,
Extremal graphs with bounded densities of small subgraphs,
\emph{J. Graph Theory} 29 (1998), 185--207.

\bibitem{KST54}
T. K\H{o}v\'{a}ri, V. T. S\'{o}s, and P. Tur\'{a}n,
On a problem of K. Zarankiewicz,
\emph{Colloq. Math.} 3 (1954), 50--57.

\bibitem{PSY26}
X. Peng, G. Song, and L.-T. Yuan,
Tur\'{a}n number of nonbipartite graphs and the Product Conjecture,
\emph{Commun. Math. Stat.} 14 (2026), 205--218.

\bibitem{Simonovits83}
M. Simonovits,
Extremal graph problems and graph products,
in \emph{Studies in Pure Mathematics: To the Memory of Paul Tur\'{a}n},
Akad\'{e}miai Kiad\'{o} and Birkh\"{a}user, 1983, 669--680.

\bibitem{Simonovits97}
M. Simonovits,
Paul Erd\H{o}s' influence on extremal graph theory,
in R. L. Graham and J. Ne\v{s}et\v{r}il (eds.),
\emph{The Mathematics of Paul Erd\H{o}s II}, Algorithms Combin. 14,
Springer, 1997, 148--192.

\bibitem{SS19}
M. Simonovits and E. Szemer\'{e}di,
Embedding graphs into larger graphs: results, methods, and problems,
in I. B\'{a}r\'{a}ny, G. O. H. Katona, and A. Sali (eds.),
\emph{Building Bridges II: Mathematics of L\'{a}szl\'{o} Lov\'{a}sz},
Bolyai Soc. Math. Stud. 28, Springer, 2019, 445--592.

\bibitem{Xu26}
C. Xu,
A finite forbidden family with superlinear surplus and non-join extremal
graphs,
arXiv:2608.02115, 2026.

\end{thebibliography}
\end{document}